\documentclass[12pt,a4paper,reqno]{amsart}
\usepackage{latexsym}
\usepackage{amssymb}
\usepackage{enumitem}

\def \ze{\varepsilon}

\def \zh{\theta}

\def \zl{\lambda}

\def \zp{\pi}

\def \zr{\rho}

\def \zf{\varphi}

\def \zq{\psi}

\def \zlma{\ell}

\def \zun{\cup}

\def \zpor{\times}

\def \zas{\ast}

\def \zmei{\leq}
\def \zmai{\geq}
\def \zco{\subset}

\def \zpe{\in}

\def \zeq{\equiv}

\def \znoi{\neq}

\def \zpar{\partial}
\def \zinf{\infty}

\def \zfl{\rightarrow}

\def \zbv{\mid}
\def \zdbv{\parallel}
\def \z/{\over}

\newtheorem*{theorem*}{Theorem}
\newtheorem{lemma}{Lemma}

\title[Smooth actions of $Aff^+ (\mathbb R)$ on compact surfaces]{Smooth actions of 
$Aff^+ (\mathbb R)$ on compact surfaces with no fixed point: 
an elementary construction}

\author{F.J.~Turiel}
\address[F.J.~Turiel]{
Geometr{\'\i}a y Topolog{\'\i}a,
Facultad de Ciencias,
Campus de Teatinos, s/n,
29071-M{\'a}laga, Spain}
\email {turiel@uma.es}

\begin{document}

\begin{abstract}
Any compact surface supports a continuous action of the orientation preserving 
affine group of the real line which is fixed point free (Lima and Plante). It is generally admitted that this
action can be taken smooth although it is not easy finding references for it. Here one gives 
a such action.
\end{abstract}

\maketitle

\section{Introduction}

All structures and objects considered are smooth that is real $C^\zinf$, unless another thing is stated.
Manifolds can have nonempty boundary.

A classical theorem by Lima \cite{Li} states that any continuous action of $\mathbb R^n$ on a compact
surface of nonzero Eule-Poincar\'e characteristic possesses fixed points. Later on Plante \cite{Pl}
extended this result to connected nilpotent Lie groups (a short proof in class $C^\zinf$
of these results can be seen in \cite{Tu}), Hirsch and Weinstein \cite{HW} to analytic actions
of connected  super-solvable Lie Groups and finally Hirsch \cite{Hi} to local nilpotent actions.

Besides Lima and Plante proved that  that every compact surface supports a  continuous fixed 
point  free action by the (orientation preserving) affine group $Aff^+ (\mathbb R)$, the first 
non-nilpotent Lie group. It belongs to the folklore that any compact surface 
(with or without boundary) supports a smooth action of $Aff^+ (\mathbb R)$ fixed point free. 
Although this fact can be deduced from other results on surfaces (for instance see \cite{Be}), 
it seems useful providing an elementary construction of such kind of actions. That is the aim 
of this work.

For this purpose, if $M$ is a compact surface it suffices finding two vector fields $X$ and $Y$ 
on it such that:
\begin{enumerate}[label={\rm (\alph{*})}]
\item $[X,Y]=-Y$.
\item $X$ and $Y$ are tangent to the boundary of $M$ if any.
\item There is no common zero of  $X$ and $Y$.
\end{enumerate}

\section{Two lemmas} 

\begin{lemma}    \label{le-1}
For each natural $k>0$ there exist real numbers 
$-1<a_{1}<...<a_{k}<1$ and two smooth functions
$\zf, \zq\colon {\mathbb R}\zfl {\mathbb R}$, where $\mathbb R$ is equipped with the 
variable $t$, such that:

\begin{enumerate}
\item $[X,Y]=-Y$ where $X=\zf {\frac{\zpar} {\zpar t}}$ 
and $Y=\zq {\frac{\zpar} {\zpar t}}\,$.
\item $\zq^{-1} (0)=\{ a_{1},...,a_{k}\}$ and $\zf(t)= -(t-a_{j})$ near each $a_{j}$. 
\item $\zf =t$ outside $(-1,1)$.
\end{enumerate}
\end{lemma}

\begin{proof} 
On ${\mathbb R}$ consider the vector fields
$A=(1+cost){\frac{\zpar} {\zpar t}}$, $B=-sint{\frac{\zpar} {\zpar t}}$ and 
$C=(1-cost){\frac{\zpar} {\zpar t}}$; then $[A,B]=-A$, $[B,C]=-C$.
Let $f\colon{\mathbb R}\zfl {\mathbb R}$ be a smooth function such that:
\smallskip

\noindent  $f(t)=0$ if $\zbv t\zbv \zmai 4k\zp$,
\smallskip

\noindent $f>0$ on $(-2\zp, 2k\zp )$,
\smallskip

\noindent $f=1+cost$ on $(2\zlma \zp -\ze, 2\zlma \zp+\ze )$, $\zlma =0,..., k-1$, for some $\ze>0$,
\smallskip 

\noindent $f=1-cost$ on $(-2\zp, -\zp]\zun [(2k-1)\zp, 2k\zp)$.
\smallskip

The vector field $Z=f{\frac{\zpar} {\zpar t}}$ is complete and its integral curve, passing through 
$-\zp$ for the time zero, runs from $-2\zp$ to $2k\zp$. So there exists a diffeomorphism 
$F\colon (-2\zp, 2k\zp)\zfl {\mathbb R}$ such that
$F_\zas Z={\frac{\zpar} {\zpar t}}$ and $F(-\zp)=0$. 

Set ${\widetilde B}= F_{*}B$ and ${\widetilde C}= F_{*}C$. Then $\widetilde C$ only vanishes at 
$F(\{ 0,2\zp,...,2(k-1)\zp\})$, ${\widetilde B}=t{\frac{\zpar} {\zpar t}}$
on $(-\zinf, 0]$, ${\widetilde B}=(t-a){\frac{\zpar} {\zpar t}}$
on $[a, \zinf)$ where $a=F((2k-1)\zp)$ and 
${\widetilde B}=-(t-F(2\zlma\zp)){\frac{\zpar} {\zpar t}}$
near each $F(2\zlma\zp)$, $\zlma =0,...,k-1$. 

Indeed, as $Z$ equals $C$ on $(-2\zp, -\zp]\zun [(2k-1)\zp, 2k\zp)$ then 
$[{\frac{\zpar} {\zpar t}},\widetilde B ]={\frac{\zpar} {\zpar t}}$ on the 
image under $F$ of this set, which implies $\widetilde B =t{\frac{\zpar} {\zpar t}}$ on
$(-\zinf,0]=F((-2\zp,-\zp])$ since $\widetilde B (0)=F_\ast (B(-\zp))=0$ and
$\widetilde B =(t-a){\frac{\zpar} {\zpar t}}$ on 
$[a,\zinf)=F( [(2k-1)\zp, 2k\zp))$ because $\widetilde B (a)=F_\ast (B((2k-1)\zp))=0$. By a similar
reason as $Z$ equals $A$ near $2\zlma\zp$, $\zlma=0,\dots,k-1$, then
$[{\frac{\zpar} {\zpar t}},\widetilde B ]=-{\frac{\zpar} {\zpar t}}$ and
$\widetilde B =-(t-F(2\zlma\zp)){\frac{\zpar} {\zpar t}}$ close to $F(2\zlma\zp)$,
$\zlma=0,\dots,k-1$.

Finally, consider a diffeomorphism $G\colon{\mathbb R}\zfl {\mathbb R}$ 
equals to $Id$ on $(-\zinf, 0]$, to 
$(t-a)$ on $[a+1, \zinf)$ and to a translation near each $F(2\zlma\zp)$, $\zlma =0,...,k-1$,
and set $X=G_{*}{\widetilde B}$, $Y=G_{*}{\widetilde C}$
and $a_{j}=G(F(2(j-1)))$, $j=1,...,k$.
\end{proof}

Endow $\mathbb R\zpor S^1$ with coordinates $(t,\zh)$.

\begin{lemma}    \label{le-2}
Consider a scalar $c\znoi 0$ and a vector field $Y$ on $(-\ze, 0)\zpor S^{1}$ such that $[X,Y]=-Y$
where $X=t^{2}e^{1/t}{\frac{\zpar} {\zpar t}}$. Then there exists a 
diffeomorphism $G\colon (-\ze,0)\zpor S^{1}\zfl (-\ze,0)\zpor S^{1}$ and two vector fields $\widetilde X$,
$\widetilde Y$ on $(-\ze,\zinf)\zpor S^{1}$ such that $G_{*}X={\widetilde X}$, 
$G_{*}Y={\widetilde Y}$ on
$(-\ze,0)\zpor S^{1}$, and ${\widetilde X}=c{\frac{\zpar} {\zpar\zh}}$,
${\widetilde Y}=0$ on $[0, \zinf)\zpor S^{1}$ 
(which implies $[\widetilde X ,\widetilde Y ]= -\widetilde Y$ everywhere).

Moreover $G$ preserves the orientation.
\end{lemma}

\begin{proof}
Since $(t^{2}e^{1/t} {\frac{\zpar} {\zpar t}})e^{-1/t}=1$ and 
$[X,{\frac{\zpar} {\zpar\zh}}]=0$ one has
$$Y= exp(-e^{-1/t})f_{1}(\zh)X +  exp(-e^{-1/t})f_{2}(\zh){\frac{\zpar} {\zpar\zh}}.$$

 Now consider the diffeomorphism $G\colon (-\ze,0)\zpor S^{1}\zfl (-\ze,0)\zpor S^{1}$ given by
$G(t,\zh )= (t, \zh + \zp(ce^{-1/t}))$, where $\zp\colon\mathbb R \zfl S^{1}\zeq{\mathbb R}/2\zp {\mathbb Z}$
is the canonical projection. Then 
$G_\zas X=t^{2}e^{1/t}{\frac{\zpar} {\zpar t}}+c{\frac{\zpar} {\zpar\zh}}$,
which is smoothly prolongated to a vector field $\widetilde X$ on $(-\ze,\zinf)\zpor S^{1}$
by setting $\widetilde X =c{\frac{\zpar} {\zpar\zh}}$ on $[0,\zinf)\zpor S^1$. On the
other hand
$$G_\zas Y=exp(-e^{-1/t})f_{1}(\zh-\zp(c e^{-1/t})){\widetilde X}  
+exp(-e^{-1/t})f_{2}(\zh-\zp(c e^{-1/t})){\frac{\zpar} {\zpar\zh}}$$
that prolongs by zero to a smooth vector field $\widetilde Y$ on $(-\ze,\zinf)\zpor S^{1}$. 

Indeed, given $g\colon S^{1}\zfl {\mathbb R}$ set 
$g^{(\zlma )}= {\frac{\zpar^\zlma g} {\zpar\zh^\zlma}}$. For every 
$k,\zlma\zpe\mathbb N$ the function $g^{(\zlma )}(\zh -\zp(c/s))$ is dominated on
${\mathbb R}^{+}\zpor S^{1}$ by $s^{-k}e^{-1/s}$ because  $g^{(\zlma )}(\zh -\zp(c/s))$ is
bounded on this set, which implies that the function
\bigskip 

$\zl(s,\zh)=
\begin{cases}
e^{-1/s}g(\zh -\zp(c/s))\quad {\rm if}\quad s>0 \\ 
0 \quad {\rm if}\quad s\zmei 0
\end{cases}$
\bigskip

\noindent is smooth on ${\mathbb R}\zpor S^{1}$.  Therefore (set $s=e^{1/t}$ if $t<0$ and 
$s=0$ if $t\zmai 0$) the functions ${\tilde f}_k$, $k=1,2$, given by
\bigskip

${\tilde f}_k(t,\zh)=
\begin{cases}
exp(-e^{-1/t})f_{k}(\zh -\zp(c e^{-1/t}))\quad {\rm if}\quad t<0 \\
0 \quad {\rm if}\quad t\zmai 0
\end{cases}$
\bigskip

\noindent are smooth on $(-\ze,\zinf)\zpor S^1$.
\end{proof} 

Consider a surface $S$ equipped with two vector fields $X_S ,Y_S$. A {\it neat hole} of 
parameter $c\znoi 0$ (on $S$ relative to $X_S ,Y_S$) is an open set $A$ of $S$ 
identified to the cylinder $(0,\zinf)\zpor S^1$  by means of a diffeomorphism
$F\colon A\zfl (0,\zinf)\zpor S^1$ in such a way that:
\begin{enumerate}[label={\rm (\arabic{*})}]
\item The ''side'' $\{\zinf\}\zpor S^1$ of the cylinder defines an end of $S$.
\item $F^{-1}\colon(0,\zinf)\zpor S^1\zfl A\zco S$ extends into a continuous map
from $[0,\zinf)\zpor S^1$ to $S$.
\item $X_S ,Y_S$ regarded on $(0,\zinf)\zpor S^1$ equal the vector fields 
$\widetilde X ,\widetilde Y$ of Lemma \ref{le-2}, that is 
$F_\ast X_S =c{\frac{\zpar} {\zpar\zh}}$ and $F_\ast Y_S =0$.
\end{enumerate} 
\medskip

{\it Truncating a neat hole} means removing the set $(b,\zinf)\zpor S^1$, for some $b>0$, 
from $A\zeq (0,\zinf)\zpor S^1$, so from $S$. Therefore the neat hole becomes a collar of
a $S^1$-component of the boundary, to which $X_S ,Y_S$ are tangent.

\section{Neat holes on $\mathbb R^2$}

On ${\mathbb R}^{2}$, with coordinates $(x_1 ,x_2 )$, when $n\zmai 1$ one considers 
the vector fields  $X=\zf_{1}(x_{1}){\frac{\zpar} {\zpar x_1}} 
+ \zf_{2}(x_{2}){\frac{\zpar} {\zpar x_2}}$ and
$Y=\zq_{1}(x_{1}){\frac{\zpar} {\zpar x_1}} 
+ \zq_{2}(x_{2}){\frac{\zpar} {\zpar x_2}}$ where
$\zf_{1}$, $\zq_{1}$ are like in Lemma \ref{le-1} for $k=1$, and
$\zf_{2}$, $\zq_{2}$ are like in Lemma \ref{le-1} for $k=n$; 
set $\zq_{1}^{-1}(0)=\{a_1 \}$ and $\zq_{2}^{-1}(0)=\{b_1 ,\dots,b_n \}$.
On the other hand if $n=0$ we consider $X=x_{1}{\frac{\zpar} {\zpar x_1}}
+x_{2}{\frac{\zpar} {\zpar x_2}}$ and $Y={\frac{\zpar} {\zpar x_1}}$.
Obviously $[X,Y]=-Y$ and $X$ is complete because (3) of Lemma 1; moreover the common 
zeros of $X,Y$ are $(a_{1},b_{1}),\dots,(a_{1},b_{n})$ if any. 
\bigskip 

{\it Constructing neat holes at $(a_{1},b_{1}),\dots,(a_{1},b_{n})$.} 

Up to translation it is enough to consider the case where 
$(a_{1},b_{j})=(0,0)$, $X=-x_{1}{\frac{\zpar} {\zpar x_1}}  
-x_{2}{\frac{\zpar} {\zpar x_2}}$ and
$Y=\zq_{1}(x_{1}){\frac{\zpar} {\zpar x_1}} 
+ \zq_{2}(x_{2}){\frac{\zpar} {\zpar x_2}}$ near the origin.

For a suitable $d>0$ identify $(d,\zinf)\zpor S^1$, endowed with coordinates $(r,\zh)$, with
a sufficiently  small punctured ball $B_\zr (0)-\{0\}$ by setting  $x_{1}=r^{-1}cos\zh$, 
$x_{2}=r^{-1}sin\zh$. Then $X=r{\frac{\zpar} {\zpar r}}$.

Let $\zf\colon{\mathbb R}\zfl {\mathbb R}$ be a smooth function such that, for some $\ze >0$:
\bigskip 

$\zf(t)=\begin{cases}
t+1\quad {\rm if} \quad -1-\ze <t<-1+\ze \\
t^{2}e^{1/t}\quad {\rm if} \quad -\ze <t<0 \\
>0 \quad {\rm if} \quad -1<t<0 \\
0\quad {\rm if} \quad t\zmai 0. 
\end{cases}$
\bigskip

Since  $\zf{\frac{\zpar} {\zpar t}}$ on $ (-1,0)$ and 
$r{\frac{\zpar} {\zpar r}}$ on ${\mathbb R}^+$
are complete and never vanish there exists a diffeomorphism 
$f\colon{\mathbb R}^{+}\zfl (-1,0)$ such that $f_\zas \left( r{\frac{\zpar} {\zpar r}}\right)
=\zf{\frac{\zpar} {\zpar t}}$. That gives rise to a diffeomorphism
$F\colon (r,\zh)\zpe {\mathbb R}^{+}\zpor S^{1}\zfl (f(r),\zh )\zpe (-1,0)\zpor S^{1}$ which
transforms $r{\frac{\zpar} {\zpar r}}$ into $\zf{\frac{\zpar} {\zpar t}}$.  

Thus for a suitable $\ze>0$ the punctured ball $B_\zr (0)-\{0\}$ can be identified with 
$(-\ze,0)\zpor S^1$, in such a way that $X=t^{2}e^{1/t}{\frac{\zpar} {\zpar t}}$.
By Lemma \ref{le-2}, attaching $(-\ze,\zinf)\zpor S^1$ to 
$\mathbb R^2 -\{(a_{1},b_{1}),\dots,(a_{1},b_{n})\}$ around $(a_1 ,b_j )$ in the obvious way
(that is by means of the diffeomorphism $G$ given by this result),
provides us with a neat hole of parameter any $c\znoi 0$.
\bigskip

{\it The neat hole at the infinity of ${\mathbb R}^2$.}

As $\zf_{1}$ and $\zf_{2}$ are bounded on $[-1,1]$ then 
$x_{1}\zf_{1}(x_{1})+x_{2}\zf_{2}(x_{2})\zmai 1$ if 
$\zdbv x\zdbv\zmai \zr$ for some $\zr$ big enough (for the case $n=0$ it is obvious). 
In other words $X$ is outwardly 
transverse to any sphere $S_{\zl}^{1}$, $\zl\zmai\zr$,
and runs to infinity. This allows us to identify ${\mathbb R}^{2}-D_{\zr}^{2}$ with 
${\mathbb R}^{+}\zpor S^{1}$, endowed with coordinates
$({\tilde r},\zh)$, by means of the integral curves of $X$ passing through $S_{\zr}^{1}$ 
for the time zero. Now
$X={\frac{\zpar} {\zpar\tilde r}}$.

Finally setting $r=e^{\tilde r}$ we have $X=r{\frac{\zpar} {\zpar r}}$ on 
$(1,\zinf)\zpor S^{1}$ with coordinates 
$(r,\zh)$. The remainder is the same as before.

\section{Construction of the action of $Aff^+ (\mathbb R)$}

The construction above of a neat hole means attaching a cylinder $(-\ze,\zinf)\zpor S^1$; 
however the space remains diffeomorphic to $\mathbb R^2 -\{(a_{1},b_{1}),\dots,(a_{1},b_{n})\}$,
$n\zmai 0$, that is to $S^2$ with $n+1$ holes, but now the vector fields on it are easily
managed. 

On $[0,\zp]\zpor\mathbb R$ endowed with coordinates $(y_1 ,y_2 )$ consider the vector fields
$X'=c{\frac{\zpar} {\zpar y_1}}$, $c\zpe\mathbb R -\{0\}$, and $Y'=0$. Now identifying
each $(0,y_2 )$ with $(\zp,-y_2 )$ gives rise to the open Moebius strip $M$ equipped with two
vector fields $X$ and $Y$ with no common zero such that $[X,Y]=-Y$. Moreover $(M-N,X,Y)$,
where $N$ is given by the condition $y_2 =0$ (before identifying), is (diffeomorphic to) a
neat hole of parameter $c$. 

On the other hand two different neat holes, belonging to the same surface or not, can be gluing 
together under the diffeomorphism 
$(t,\zh)\zpe\mathbb R^+ \zpor S^1 \zfl (t^{-1},h(\zh))\zpe\mathbb R^+ \zpor S^1$ where $h$ belongs
to the orthogonal group $O(2)$. When the parameter of each neat hole may be independently
chosen , it is always possible at the same time, by gluing the vector fields also, to construct two 
vector fields  $\overline X ,\overline Y$ with no common zero such that 
$[\overline X ,\overline Y ]=-\overline Y$.

Finally, considering $S^2$ with $n+1$ neat holes constructed as before for a suitable $n$, 
truncating some neat holes (only if the boundary is nonempty), gluing together or not some 
pairs of neat holes (perhaps some of them reversing the orientation), gluing or not an open disk 
(that is $\mathbb R^2$ with the neat hole at the infinity)  or an open Moebius strip, leads to construct 
two vector fields $\widehat X ,\widehat Y$ on any connected compact surface, with no common 
zero and tangent to the boundary if any, such that   $[\widehat X ,\widehat Y]=-\widehat Y$. 

Therefore an action of $Aff^+ (\mathbb R)$ with no fixed
point can be constructed on any compact surface.


\end{document}